\documentclass{article}
\usepackage[utf8]{inputenc}
\usepackage{tasks}
\usepackage{amssymb}
\usepackage{amsmath}
\usepackage{xcolor,comment}
\usepackage{multicol}
\usepackage{setspace}
\usepackage{amsthm}
\usepackage{standalone}
\usepackage{graphicx}
\usepackage{tikz, float}
\usepackage{authblk}
\usepackage{caption}
\usetikzlibrary{decorations.pathreplacing, calligraphy}
\usepackage{url}
\usepackage{lipsum}
\newtheorem{theorem}{Theorem}
\newtheorem{definition}{Definition}
\newtheorem{lemma}[theorem]{Lemma}
\newtheorem{corollary}{Corollary}[theorem]

\theoremstyle{definition}

\theoremstyle{definition}

\newcounter{casecount}
\newenvironment{case}{\refstepcounter{casecount}\textbf{Case \arabic{casecount}:}}{}

\usepackage{etoolbox}
\AtBeginEnvironment{proof}{\setcounter{casecount}{0}}

\newcommand{\Z}{\mathbb{Z}}
\newcommand{\X}{\mathbf{X}}

\tikzstyle{point}=[circle,draw, fill=black, scale = 0.3]

\title{A Classification of Winning Sets of Cops in $\mathbb{Z}^n$}

\author[$\dagger$]{Kenzie Fontenot}
\author[$\star$]{Iris Nguyen}
\author[$\ddag$]{Cody Olsen}

\affil[$\dagger$]{Department of Mathematics, University of North Texas, Denton, Texas 76203}
\affil[$\star$]{Department of Mathematics, University of Texas at Austin, Austin, Texas 78712}
\affil[$\ddag$]{Department of Mathematics, University of Texas at San Antonio, San Antonio, Texas 78249}

\begin{document}

\maketitle

\begin{abstract}
The game of Cops and Robbers is a pursuit-evasion game on graphs that has been extensively studied in finite settings, particularly through the concept of cop number. In this paper, we explore infinite variants of the game, focusing on the lattice graph $\mathbb{Z}^n$. Since the cop number of $\mathbb{Z}^n$ is infinite, we shift attention to the notion of \emph{cop density}, examining how sparsely cops may be placed while still guaranteeing capture. We introduce the framework of \emph{coordinate matching} as a central strategy and prove the existence of density-zero configurations of cops that ensure eventual capture of the robber. Building on this, we provide a complete classification of winning cop sets in $\mathbb{Z}^n$, showing necessary and sufficient conditions for capture based on infinite distributions of cops across coordinate directions. We conclude with directions for further study, including variations of the game where the robber imposes spatial restrictions on cop placement.
\end{abstract}

\newpage

\section{Introduction}

The game of Cops and Robbers was established in 1978 by French mathematician Alain Quilliot in his PhD dissertation \cite{Quillot78}. Cops and Robbers is a turn based pursuit-evasion game played by two players on a graph $G$: one player representing a singular robber and the other controlling a set of cops. The goal of the player representing the robber is to evade capture by the cops indefinitely while the role of the player controlling the cops is to, after some finite number of moves, capture the robber.

First, we will introduce a few important terms of basic graph theory. A \textit{graph} $G$ consists of a set of \textit{vertices} $V(G)$ and a set of pairs of vertices $E(G)$, called \textit{edges}. If a pair of vertices $(u,v)\in E(G)$, then we say the vertices $u$ and $v$ are \textit{adjacent} and will sometimes refer to $u$ and $v$ as neighbors.

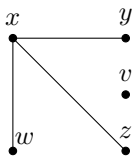
\begin{figure}[H]
\begin{center}

    \begin{tikzpicture}[scale=1.5]
        \node at (0,0) [point, label = $x$] {};

        \draw[color = black] (0,0) -- (1,0);
        \node at (1,0) [point, label = $y$, color = black] {};
        
        \draw[color = black] (0,0) -- (1,-1);
        \node at (1,-1) [point, label = $z$, color = black] {};

        \draw[color = black] (0,0) -- (0,-1);
        \node at (0,-1) [point, label={[xshift=4.5 pt, yshift= -2.5pt]$w$}, color = black] {};

        \node at (1,-0.5) [point, label = $v$] {};
    \end{tikzpicture}

    \label{Neighbors}

    \caption{The neighbors of $x$ are $w,y$ and $z$, since they connect to $x$ via an edge. $v$ is not a neighbor to $x$ since there is no edge connecting them.}
 \end{center}   
\end{figure}

To begin the game of Cops and Robbers on a graph $G$, the player controlling the cops will place each cop on a vertex of $G$, with multiple cops being allowed to occupy the same vertex. Next, the player controlling the robber selects a vertex of $G$ and places himself there. After this initial step, the players controlling the cops and robbers alternate taking turns, with the robber moving first after both players have picked their starting vertices. We assume any vertex is adjacent to itself, which allows a player to remain on the same vertex on two consecutive turns. For the sake of brevity, we will say the cops take some action instead of saying the player controlling the cops takes that action, and use a similar convention for the robber. During the cops' turn, each cop may move to any vertex which is adjacent to the one they are occupying, and the robber can do the same on his turn. The cops win if, at the end of their turn, one of the cops occupies the same vertex as the robber. The robber wins if he is able to evade capture indefinitely \cite{Bonato11}.

The \textit{cop number} of a graph $G$, denoted $C(G)$, is defined as the minimum number of cops required to successfully capture the robber on $G$ \cite{Aigner84}. If such a number does not exist, then we say $C(G)=\infty$. If $G$ is a finite graph with $n$ vertices, then $C(G) \leq n$, since the cops could occupy all the vertices of $G$ and immediately catch the robber when he places himself. Most of the study around cops and robbers is devoted to finding cop numbers of graphs.

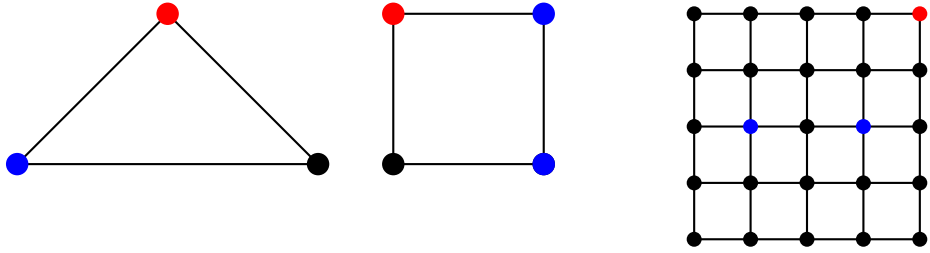
\begin{figure}
\begin{tikzpicture}

    \draw[thick] (-1,0) -- (-3,-2);
    \draw[thick] (-3,-2) -- (1,-2);
    \draw[thick] (1,-2) -- (-1,0);
    
    \node[circle,fill=red, inner sep=3pt,label=] at (-1,0){};
    
    \node[circle,fill=blue, inner sep=3pt,label=] at (-3,-2){};

    \node[circle,fill=black, inner sep=3pt,label=] at (1,-2){};

    \draw[thick] (2,0) -- (4,0);
    \draw[thick] (4,0) -- (4,-2);
    \draw[thick] (4,-2) -- (2,-2);
    \draw[thick] (2,-2) -- (2,0);
    
    \node[circle,fill=black, inner sep=3pt,label=] at (2,-2){};
    \node[circle,fill=black, inner sep=3pt,label=] at (4,-2){};
    
    \node[circle,fill=red, inner sep=3pt,label=] at (2,0){};
    
    \node[circle,fill=blue, inner sep=3pt,label=] at (4,0){};
    \node[circle,fill=blue, inner sep=3pt,label=] at (4,-2){};

    \draw[thick] (6,0) -- (9,0);
    \draw[thick] (6,-.75) -- (9,-.75);
    \draw[thick] (6,-1.5) -- (9,-1.5);
    \draw[thick] (6,-2.25) -- (9,-2.25);
    \draw[thick] (6,-3) -- (9,-3);
    \draw[thick] (6,0) -- (6,-3);
    \draw[thick] (6.75,0) -- (6.75,-3);
    \draw[thick] (7.5,0) -- (7.5,-3);
    \draw[thick] (8.25,0) -- (8.25,-3);
    \draw[thick] (9,0) -- (9,-3);
    
    \node[circle,fill=red, inner sep=2pt,label=] at (9,0){};
    
    \node[circle,fill=blue, inner sep=2pt,label=] at (6.75,-1.5){};
    \node[circle,fill=blue, inner sep=2pt,label=] at (8.25,-1.5){};
    
    \node[circle,fill=black, inner sep=2pt,label=] at (6,0){};
    \node[circle,fill=black, inner sep=2pt,label=] at (6,-.75){};
    \node[circle,fill=black, inner sep=2pt,label=] at (6,-1.5){};
    \node[circle,fill=black, inner sep=2pt,label=] at (6,-2.25){};
    \node[circle,fill=black, inner sep=2pt,label=] at (6,-3){};
    \node[circle,fill=black, inner sep=2pt,label=] at (6.75,0){};
    \node[circle,fill=black, inner sep=2pt,label=] at (6.75,-.75){};
    \node[circle,fill=black, inner sep=2pt,label=] at (6.75,-2.25){};
    \node[circle,fill=black, inner sep=2pt,label=] at (6.75,-3){};
    \node[circle,fill=black, inner sep=2pt,label=] at (7.5,0){};
    \node[circle,fill=black, inner sep=2pt,label=] at (7.5,-.75){};
    \node[circle,fill=black, inner sep=2pt,label=] at (7.5,-1.5){};
    \node[circle,fill=black, inner sep=2pt,label=] at (7.5,-2.25){};
    \node[circle,fill=black, inner sep=2pt,label=] at (7.5,-3){};
    \node[circle,fill=black, inner sep=2pt,label=] at (8.25,0){};
    \node[circle,fill=black, inner sep=2pt,label=] at (8.25,-.75){};
    \node[circle,fill=black, inner sep=2pt,label=] at (8.25,-2.25){};
    \node[circle,fill=black, inner sep=2pt,label=] at (8.25,-3){};
    \node[circle,fill=black, inner sep=2pt,label=] at (9,-1.5){};
    \node[circle,fill=black, inner sep=2pt,label=] at (9,-.75){};
    \node[circle,fill=black, inner sep=2pt,label=] at (9,-2.25){};
    \node[circle,fill=black, inner sep=2pt,label=] at (9,-3){};

\end{tikzpicture}
\caption{Illustration of winning configurations of cops on $C_3$ (A triangle), $C_4$ (a square) and $P_5 \times P_5$ (a five-by-five grid).}
\label{Examples}
\end{figure}

Figure \ref{Examples} shows three examples of graphs that are all finite and thus all have a finite cop number, namely $C(C_3)=1$, $C(C_4)=2$, and $C(P_5\times P_5)=2$. Infinite graphs, or graphs where $|V(G)|=\infty$, are studied less frequently, but are nevertheless fascinating. While every finite graph has a finite cop number, infinite graphs can have either finite or infinite cop numbers. For example: any infinite graph with a vertex that neighbors all other vertices will have cop number one. For the remainder of this paper, our attention will be on the graphs $\mathbb{Z}^n$, where $V(\mathbb{Z}^n)=\{(x_1,x_2,\dots,x_n):x_i\in \mathbb{Z} \text{ for each }1\leq i \leq n\}$ and $E(\mathbb{Z}^n)=\{{(u,v)\in \mathbb{Z}^n\times \mathbb{Z}^n: ||u-v||_1=1\}}$.

Notice that on $G = {Z}^n$, it is clear to see that $C(G)=\infty$. Since the cops' placement is decided first, if there are only finitely many cops the robber could simply place himself away from all of the cops. 


Since the cop number of $\Z^n$ must be infinity, the interesting question becomes what is the "smallest" set of cops that can still catch the robber? To explore this question, we turn to the concept of \textit{density}. Let us recall the standard definition of density in $\mathbb{Z}^n$: 

\begin{definition}
     If $A$ is a subset of $\Z^n$, the density of A, denoted \( d_A \), is defined as $\displaystyle{\lim_{m\to \infty}\frac{|A \cap [-m,m]^n | }{|[-m,m]^n|}}$. The expression after the limit represents the ratio of how many elements of A occur in a given area.
\end{definition}

If we placed a cop on every vertex of $\Z^n$, this would be a set with the maximum density of 1. Figure \ref{density 1/2} gives a visual representation of how to compute the density of a set of cops, $\mathcal{C}$, that occupy the entire upper half plane in $\mathbb{Z}\times \mathbb{Z}$. The dashed lines demonstrate the squares $[-m,m]^2$ in $\mathbb{Z}\times \mathbb{Z}$ while the filled in vertices represent the starting placement of the cops. For the sake of visual simplicity, we will omit from drawing any vertices in $\mathbb{Z}\times\mathbb{Z}$ that are not currently occupied by either a cop or a robber for the remainder of this paper. It is clear that $\displaystyle{\lim_{m\to \infty}\frac{|\mathcal{C} \cap [-m,m]^2 | }{|[-m,m]^2|}}=\displaystyle{\lim_{m\to \infty}\frac{(\frac{m+1}{2})m }{m^2}}=\displaystyle{\lim_{m\to\infty} \frac{m^2+m}{2m^2}}=\displaystyle{\frac{1}{2}}$. Given a collection of cops' $\mathcal{C}$ with the set of starting positions $\mathcal{A}$, define the \textit{cop density} of $\mathcal{C}$ to be $d_\mathcal{A}$. For simplicity, when discussing the density of a collection of cops' starting positions, we will simply say $\mathcal{C}=\mathcal{A}$ and note that the density of a set of cops will refer to the density of the collection of vertices that the cops initially occupy. For the remainder of this paper, we will assume $G=\mathbb{Z}^n$ unless otherwise specified and we will drop mentions of the graph $G$, as it is understood to be $\mathbb{Z}^n$.

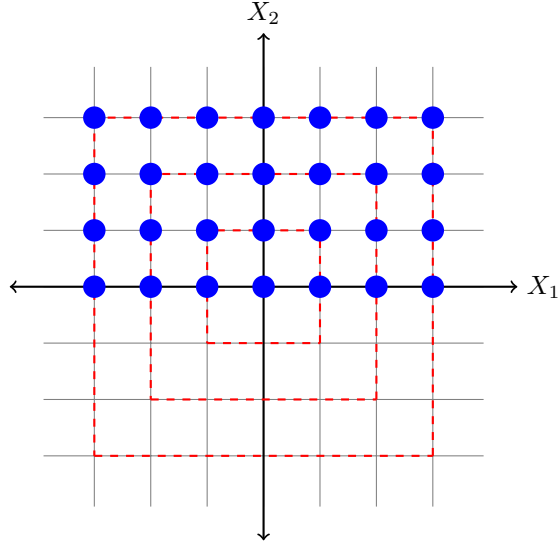
\begin{figure}
\begin{center}
\begin{tikzpicture}[scale=0.75]
\label{density}
\draw[step=1cm,gray,thin] (-3.9,-3.9) grid (3.9,3.9);

\draw[thick,->] (0,0) -- (4.5,0);
\draw[thick,->] (0,0) -- (0,4.5);
\draw[thick,<-] (-4.5,0) -- (0,0);
\draw[thick,<-] (0,-4.5) -- (0,0);

\draw[red,thick,dashed] (-1,-1) -- (1,-1);
\draw[red,thick,dashed] (1,-1) -- (1,1);
\draw[red,thick,dashed] (1,1) -- (-1,1);
\draw[red,thick,dashed] (-1,1) -- (-1,-1);

\draw[red,thick,dashed] (2,2) -- (-2,2);
\draw[red,thick,dashed] (-2,2) -- (-2,-2);
\draw[red,thick,dashed] (-2,-2) -- (2,-2);
\draw[red,thick,dashed] (2,-2) -- (2,2);

\draw[red,thick,dashed] (-3,-3) -- (3,-3);
\draw[red,thick,dashed] (3,-3) -- (3,3);
\draw[red,thick,dashed] (3,3) -- (-3,3);
\draw[red,thick,dashed] (-3,3) -- (-3,-3);

\node[circle,fill=blue, inner sep=3pt,label=] at (0,1){};
\node[circle,fill=blue, inner sep=3pt,label=] at (0,2){};
\node[circle,fill=blue, inner sep=3pt,label=] at (0,3){};
\node[circle,fill=blue, inner sep=3pt,label=] at (0,0){};
\node[circle,fill=blue, inner sep=3pt,label=] at (1,0){};
\node[circle,fill=blue, inner sep=3pt,label=] at (2,0){};
\node[circle,fill=blue, inner sep=3pt,label=] at (3,0){};
\node[circle,fill=blue, inner sep=3pt,label=] at (-1,0){};
\node[circle,fill=blue, inner sep=3pt,label=] at (-2,0){};
\node[circle,fill=blue, inner sep=3pt,label=] at (-3,0){};
\node[circle,fill=blue, inner sep=3pt,label=] at (1,1){};
\node[circle,fill=blue, inner sep=3pt,label=] at (1,2){};
\node[circle,fill=blue, inner sep=3pt,label=] at (1,3){};
\node[circle,fill=blue, inner sep=3pt,label=] at (2,1){};
\node[circle,fill=blue, inner sep=3pt,label=] at (2,2){};
\node[circle,fill=blue, inner sep=3pt,label=] at (2,3){};
\node[circle,fill=blue, inner sep=3pt,label=] at (3,1){};
\node[circle,fill=blue, inner sep=3pt,label=] at (3,2){};
\node[circle,fill=blue, inner sep=3pt,label=] at (3,3){};
\node[circle,fill=blue, inner sep=3pt,label=] at (-1,1){};
\node[circle,fill=blue, inner sep=3pt,label=] at (-1,2){};
\node[circle,fill=blue, inner sep=3pt,label=] at (-1,3){};
\node[circle,fill=blue, inner sep=3pt,label=] at (-2,1){};
\node[circle,fill=blue, inner sep=3pt,label=] at (-2,2){};
\node[circle,fill=blue, inner sep=3pt,label=] at (-2,3){};
\node[circle,fill=blue, inner sep=3pt,label=] at (-3,1){};
\node[circle,fill=blue, inner sep=3pt,label=] at (-3,2){};
\node[circle,fill=blue, inner sep=3pt,label=] at (-3,3){};

\draw(4.5,0) node[anchor=west] {$X_{1}$};
\draw(0,4.5) node[anchor=south] {$X_{2}$};

\end{tikzpicture}
\caption{The above set of cops occupying the upper half plane has density ${\frac{1}{2}}$.}
\label{density 1/2}
\end{center}
\end{figure}

Another interesting factor that comes along with the discussion of cop density is the initial placement of the cops on $\mathbb{Z}^n$. In particular, the initial placement of the cops on the graph matters. This is true even if only a specific density of cops is allowed to be placed on the graph. For example, consider $\Z^2$ and the specified density $\frac{1}{2}$. We saw in Figure \ref{density 1/2} that the collection of cops, $\mathcal{C}_1$, occupying the upper half plane has density $\frac{1}{2}$. Similarly the collection of cops $\mathcal{C}_2=\{(x,y):x\in \mathbb{Z}, y\in 2\mathbb{Z}\}$ has density $\frac{1}{2}$ in $\mathbb{Z}^2$ as well. Notice that $\mathcal{C}_1$ is unable to capture the robber, as the robber can pick his starting position to be $(0,-2)$ and move to $(0,-2-k)$ on each turn $k$. However $\mathcal{C}_2$ is clearly able to catch the robber, regardless of his starting position. Thus we see that even given a specified positive density of cops, the initial configuration of cops matters.

\section{Coordinate Matching}

\begin{definition}
    On $\Z^n$ with axes $X_1,\dots,X_n$, let the robber's position at a given time be denoted $(x_1,x_2,\dots,x_n)$ and a cop's position be denoted $(y_1,y_2,\dots,y_n)$. We say a cop has \text{coordinate matched} the robber if, at the end of the cop's turn, $x_k= y_k$ for all $k\neq i$ for some $i\in \{1,\dots,n\}$.

\end{definition}

The notion of coordinate matching will be integral to all of our strategies discussed in this paper for catching the robber. The idea being that once a cop has coordinate matched a robber, that cop's job will be to prevent the robber from escaping by traveling in either the positive or negative $X_i$ direction, depending on which side of the robber the cop is on with respect to $X_i$. Our goal will be to coordinate match the robber on each side of each axis, guaranteeing that the robber has no direction to escape in. Let $c_i^\pm$ denote the $i$th coordinate of $C_i^\pm$ and suppose the robber's position is $(x_1,x_2\,\dots,x_n)$. We say the robber has been \textit{matched along $X_i$} if there exist two cops, $C_i^+$ and $C_i^-$ such that both $C_i^+$ and $C_i^-$ have both coordinate matched with $c_{j}^\pm=x_j$ for all $j\neq i$ and $c_i^-\leq x_i\leq c_i^+$.

\begin{theorem}
    On $\Z^n$, there is a density zero set of cops which can catch a robber.

\end{theorem}

\begin{proof}

    For each $i\in \{1,\dots,n\}$, let $C_{i,m}^+$ and $C_{i,m}^-$ be cops whose starting positions are $(c_{i,1}^+,c_{i,2}^+,\dots,c_{i,n}^+)$ and $(c_{i,1}^-,c_{i,2}^-,\dots,c_{i,n}^-)$, where $c_{i,k}^\pm=0$ for all $k\neq i$ and $c_{i,i}^+=2^m=-c_{i,i}^-$ for $m\geq 1$. Define the set of cops to be $\mathcal{C}=\{C_{i,m}^+,C_{i,m}^-:1\leq i\leq n, m\in \mathbb{N}\}$. With this collection of cops $\mathcal{C}$, the intuition is that given any starting position the robber chooses, there are cops located on each axis on either side of the robber who can coordinate match the robber along their respective axis.

\begin{figure}
\begin{center}
\begin{tikzpicture}[scale=0.4]

\draw[step=1cm,gray,thin] (-8.9,-8.9) grid (8.9,8.9);

\draw[thick,->] (0,0) -- (8.9,0);
\draw[thick,->] (0,0) -- (0,8.9);
\draw[thick,<-] (-8.9,0) -- (0,0);
\draw[thick,<-] (0,-8.9) -- (0,0);

\node[circle,fill=blue, inner sep=3pt,label=$C_{1,1}^+$] at (2,0){};
\node[circle,fill=blue, inner sep=3pt,label=$C_{1,1}^-$] at (-2,0){};
\node[circle,fill=blue, inner sep=3pt,label={[anchor=east]$C_{2,1}^+$}] at (0,2){};
\node[circle,fill=blue, inner sep=3pt,label={[anchor=east]$C_{2,1}^-$}] at (0,-2){};

\node[circle,fill=blue, inner sep=3pt,label=$C_{1,2}^+$] at (4,0){};
\node[circle,fill=blue, inner sep=3pt,label=$C_{1,2}^-$] at (-4,0){};
\node[circle,fill=blue, inner sep=3pt,label={[anchor=east]$C_{2,2}^+$}] at (0,4){};
\node[circle,fill=blue, inner sep=3pt,label={[anchor=east]$C_{2,2}^-$}] at (0,-4){};

\node[circle,fill=blue, inner sep=3pt,label=$C_{1,3}^+$] at (8,0){};
\node[circle,fill=blue, inner sep=3pt,label=$C_{1,3}^-$] at (-8,0){};
\node[circle,fill=blue, inner sep=3pt,label={[anchor=east]$C_{2,3}^+$}] at (0,8){};
\node[circle,fill=blue, inner sep=3pt,label={[anchor=east]$C_{2,3}^-$}] at (0,-8){};

\draw(8.9,0) node[anchor=west] {$X_{1}$};
\draw(0,8.9) node[anchor=south] {$X_{2}$};

\end{tikzpicture}
\caption{Starting configuration of $\mathcal{C}$ in $\mathbb{Z}^2$.}
\end{center}
\end{figure}
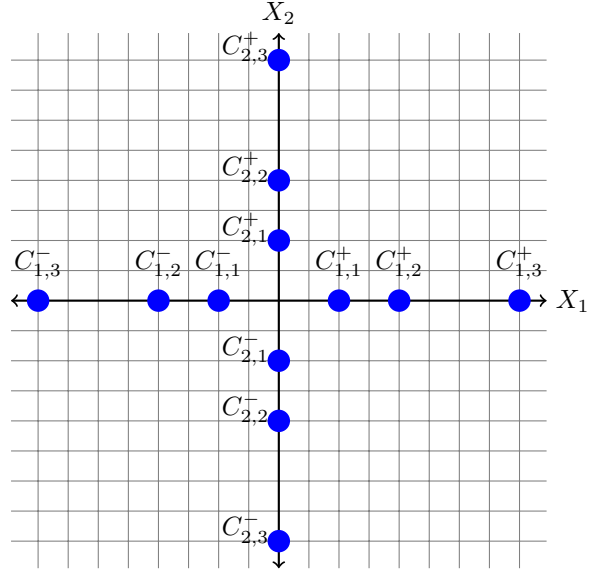
\label{Starting configuration}

   To see this, assume that the robber's starting position is \( (x_1, x_2, ... ,x_n) \). Let $N$ be the least integer such that $2^N>\displaystyle{\sum_{j=1}^{n} |x_j|}$. We claim the collection of cops $\mathcal{C}'=\{C_{i,N}^\pm:1\leq i\leq n\}\subset\mathcal{C}$ can catch the robber. The idea is that $C^{+}_{i,N}$ would be able to coordinate match with respect to $X_i$ if the robber ran directly in the positive $i$ direction and similarly $C_{i,N}^-$ would be able to coordinate match the robber with respect to $X_i$ if he ran directly in the negative $i$ direction.

     Our strategy for $\mathcal{C}$ will first be that for every cop $C\in \mathcal{C}\setminus \mathcal{C}'$, $C$ will remain on his starting vertex on each turn. Thus the only cops that will move during this run of the game will be the cops in $\mathcal{C}'$. For ease of notation since $N$ is now fixed, we will denote $C_{i,N}^+$ as $C_{i}^+$ and $C_{i,N}^-$ as $C_{i}^-$. Let $C_{i}^{+}$'s coordinates at the end of turn $k$ be given by $(c_{i,1}^{+}(k), c_{i,2}^{+}(k),\dots, c_{i,n}^{+}(k))$ and similarly for $C_{i}^{-}$. Define $D_{i,k}^{+} = \sum_{j=1}^n |c_{i,j}^{+}(k) - x_j(k)|$ and define $D_{i,k}^{-}$ similarly. Note that if $D_{i,k}^{\pm}$ is 0, then $C_i^\pm$ has caught the robber. 
      On turn $k$, we assume the following induction assumption holds for each cop $C_{i}^{\pm}$:


       \begin{enumerate}
        \item $\displaystyle c_{i,i}^{+}(k) - x_i(k) > \sum_{j\neq i}|x_j(k)-c^{+}_{i,j}(k)|$

       \item  $x_i(k) -\displaystyle c^{-}_{i,i}(k) > \sum_{j\neq i}|x_{j}(k)-c^{-}_{i,j}(k)|$

        \item $D_{i,k}^{\pm} \leq D_{i,k-1}^{\pm}$, and for at least one cop, this inequality is strict.
    \end{enumerate}




       We note that (1) and (2) imply that $c_{i,i}^{-}(k) < x_i(k) < c_{i,i}^{+}(k)$, meaning the robber will be between the two cops when they do coordinate-match him along $X_i$. If we can define a strategy which satisfies these conditions, then eventually the cops will be in a position where they have coordinate matched the robber along each axis. From the way we picked the cops' coordinates, conditions (1) and (2) are satisfied, and (3) is vacuously true.


     At the beginning of turn $k+1$, suppose the robber starts on $ (x_1(k), x_2(k), ... , x_n(k)) $ and moves parallel to $\mathbf{X}_i$. Without loss of generality, assume he moves from $(x_1(k), x_2(k), \dots, x_n(k))$ to $(x_1(k)+1,x_2(k),\dots, x_n(k))$. In this scenario, for $j\neq 1$, $C_j^+$ and $C_j^-$ should both increase their first coordinate by 1. Thus, for any cops other than $C_1^{+}$ and $C_1^{-}$, $|x_j(k+1) - c_{j,j}^{+}(k+1)| = |x_j(k) - c_{j,j}^{+}(k)|$ for all $j\neq 1$. This means (1) and (2) are satisfied for these cops.
    
    

    Inductively, $\displaystyle c^{+}_{i,i}(k) - x_i(k) > \sum_{j\neq i}|x_j(k)-c^{+}_{i,j}(k)|$. If the sum on the right is zero, then $C_i^{+}$ has coordinate-matched the robber and should move closer to him on the first coordinate. Otherwise, have $C_1^{+}$ move closer to the robber on axis $\X_j$, where $j\neq 1$ is least such that $C_1^+$'s $j$th coordinate differs from the robber's. In this case, $c_{1,1}^{+}(k+1) - a_1(k+1) =  c^{+}_{1,1}(k) - a_1(k) -1$, $a_j(k+1) - c_{1,j}^{+}(k+1) = x_j(k) - c_{1,j}^{+}(k)-1$, and $x_l(k+1) - c_{1,l}^{+}(k+1) = x_l(k) - c_{1,l}^{+}(k)$ for all $l\neq 1,j$. Thus, property (1) holds at step $k$.


      Repeat this argument for $C_1^{-}$. The only difference here is that $c_{1,1}^{-}(k+1) - x_1(k+1) =  c^{-}_{1,1}(k) - x_1(k) + 1$, while $x_j(k+1) - c_{1,j}^{-}(k+1) = x_j(k) - c_{1,j}^{-}(k)-1$, meaning the inequality in property (2) has grown. Using this strategy, the inequality in (3) is true for each cop, and is a strict inequality for $C_1^{+}$. This strategy must eventually terminate, since $D_{i,k}^{\pm}$ will be 0 on some turn $k$.  
    

\end{proof}

In $\mathbb{Z}^n$, with axes $X_1,X_2,...,X_n$, we define for each $m\in \{1,...,n\}$ the sets $V_{m,l}^+=\partial\{ (x_1,x_2,...,x_n)\subset \mathbb{R}^n: x_m\geq|x_i|+l,i\neq m\}\cap \mathbb{Z}^n$ and $V_{m,l}^-=\partial \{  (x_1,x_2,...,x_n)\subset \mathbb{R}^n: x_m\leq-|x_i|-l,i\neq m\}\cap \mathbb{Z}^n$, where $\partial$ denotes the boundary of a set. In other words, each of the sets describes the boundary of an $n$-dimensional pyramid centered around the positive or negative side of the $X_m$ axis shifted along $X_m$ $l$ units. 

\begin{lemma}\label{Origin catch in n}
    Suppose the robber's starting position in $\mathbb{Z}^n$ is $(0,0,\dots,0)$. If for each $1\leq i \leq n$ there exist $k(i)^+,k(i)^-\in \mathbb{N}$ such that there is at least one cop in each of $V_{i,k(i)^+}^+$ and $V_{i,k(i)^-}^-$, say $C_i^+$ and $C_i^-$ where $C_i^\pm\neq C_j^\pm$ for $i\neq j$, then the collection of cops can catch the robber. 
\end{lemma}

\begin{proof}
    Let $\mathcal{C}$ denote the set of cops and suppose for each $1\leq i \leq n$ there exist $k(i)^+,k(i)^-\in \mathbb{N}$ such that there is at least one cop in each of $V_{i,k(i)^+}^+$ and $V_{i,k(i)^-}^-$, say $C_i^+$ and $C_i^-$ where $C_i^+\neq C_j^\pm$ for $i\neq j$, $C_i^-\neq C_j^\pm$ for $i\neq j$, and $C_i^+\neq C_i^-$. 

    Assume that the robber's starting location is \( (0, 0,\dots , 0) \). Let $\mathcal{C}'=\{C_i^+,C_i^-:1\leq i \leq n \}$. We will define a strategy for $\mathcal{C}$ as follows: if $C\in \mathcal{C}\setminus \mathcal{C}'$, then $C$ will remain on his starting position on each turn. Suppose on turn $k$ the robber moves parallel to $X_i$ in the positive direction. Then $C_j^\pm$ will move parallel to $X_i$ in the positive direction for all $j\neq i$. If $C_i^+$ has not coordinate matched the robber, then there exists some $j\neq i$ such that the robber's $j$th coordinate and the cop's $j$th coordinate, $x_j$ and $c_{1,j}^+$, do not match. If $c_{1,j}^+>x_j$, then $C_i^+$ will move in the negative direction parallel to $X_j$ and if $c_{1,j}^+<x_j$, then $C_i^+$ will move in the positive direction parallel to $X_j$.  If $C_i^+$ has coordinate matched the robber, then he will move parallel to the $X_i$ axis in the direction of the robber. Define the strategy similarly for $C_i^-$. 
    
   At the end of turn $k$, let $C_i^{+}$'s coordinates be denoted $C_i^+(k)=(c_{i,1}^{+}(k), c_{i,2}^{+}(k),\dots,\\
    \hspace{2 ex}c_{i,n}^{+}(k))$ and similarly for $C_i^{-}$. Let $D_{i,k}^{\pm} = \sum_{j=1}^n |c_{i,j}^{\pm}(k) - x_j(k)|$, where, on turn $k$, the robber's coordinates are given by $R(k)=(x_1(k),x_2(k),\dots,x_n(k))$. Note that if $D_{i,k}^{\pm}=0$, then $C_i^\pm$ has caught the robber. On turn $k$, we assume the following induction assumption holds for each cop:

    \setlength{\abovedisplayskip}{0pt}
\setlength{\belowdisplayskip}{2pt}
\setlength{\abovedisplayshortskip}{2pt}
\setlength{\belowdisplayshortskip}{2pt}

    \begin{itemize} 
        \item $D_k^{i\pm} \leq D_{k-1}^{i\pm}$, \text{and this inequality is strict for at least one } $i\in \{1,\dots,n\}$
         \end{itemize}

\vspace{3mm}
    
  For the base case with $k=0$ our induction hypothesis is vacuously true. Now suppose the robber moves on some turn $k$. Without loss of generality due to symmetry, suppose the robber moves parallel to the $X_1$ -axis in the positive direction, i.e. $R(k+1)=(x_{i,1}^+(k)+1,x_{i,2}^+(k),\dots,x_{i,n}^+(k))$. According to our strategy, for all $j\neq 1$, $C_j^\pm(k)=(c_{i,1}^+(k)+1,c_{i,2}^+(k),\dots,c_{i,n}^+(k))$ and therefore

\setlength{\abovedisplayskip}{-5pt}
\setlength{\belowdisplayskip}{5pt}
\setlength{\abovedisplayshortskip}{0pt}
\setlength{\belowdisplayshortskip}{0pt}

    \begin{align*}     
      D_{j,k+1}^\pm &= \sum_{j=1}^n |c_{i,j}^{\pm}(k+1) - x_j(k+1)|\\
      &=|c_{i,1}(k+1)-x_1(k+1)|+\sum_{j\neq 1}^n |c_{i,j}^{\pm}(k+1) - x_j(k+1)|\\
      &=|c^\pm_{i,1}(k)+1-(x_1(k)+1)|+\sum_{j\neq 1}^n |c_{i,j}^{\pm}(k) - x_j(k)|\\
      &= \sum_{j=1}^n |c_{i,j}^{\pm}(k) - x_j(k)|=D_{j,k}^\pm. 
    \end{align*}

  Lastly, we will show that $D_{1,k+1}^+<D_{1,k}^+$ and $D_{1,k+1}^-=D_{1,k}^-$. For $C_1^-$ we will consider the following cases:

  \begin{case} Suppose there is some $j\neq 1$ such that $c_{1,j}^-(k)\neq x_j(k)$ for some $j\neq 1$, i.e. $C_1^-$ has not coordinate-matched the robber. Without loss of generality suppose $c_{1,j}^-(k)>x_j(k)$. Then $c_{1,j}^-(k+1)=c_{1,j}^-(k)-1$ and $c_{1,i}^-(k+1)=c_{1,i}^-(k)$ for all $i\neq j$. Therefore:

  \setlength{\abovedisplayskip}{-5pt}
\setlength{\belowdisplayskip}{5pt}
\setlength{\abovedisplayshortskip}{0pt}
\setlength{\belowdisplayshortskip}{0pt}

     \begin{align*}  
     D_{1,k+1}^-&= \sum_{m=1}^n |c_{1,m}^{-}(k+1) - x_m(k+1)|\\
        &=|c^-_{1,j}(k+1)-x_j(k+1)|+\sum_{m\neq j}^n |c_{1,m}^{-}(k+1) - x_m(k+1)|\\
        &=|c^-_{1,j}(k)-1-x_j(k)|+|c_{1,1}^-(k)-(x_1(k)+1)|+\sum_{m\neq 1,j}^n |c_{1,m}^{-}(k) - x_m(k)|\\
        &=\left[c^-_{1,j}(k)-1 - x_j(k)\right] +\left[x_1(k)+1-c_{1,1}^-(k)\right]+\sum_{m\neq 1,j}^n |c_{1,m}^{-}(k) - x_m(k)|\\
        &=D_{1,k}^-
\end{align*}
    \end{case}

\textbf{}
  \begin{case}
      Suppose $C_1^+$ has coordinate matched the robber. Then $x_j(k)=c_{1,j}^+(k)$ for all $j\neq 1$. If $x_1(k)=c_{1,1}^+(k)$ as well, then the robber has been caught. So let us suppose $x_1(k)< c_{1,1}^+(k)$. Then according to $\sigma$, $c_{1,1}^+(k+1)=c_{1,1}^-(k)-1$ and $c_{1,j}^+(k+1)=c_{1,j}^-(k)$ for all $j \neq 1$. A computation similar to the above case easily shows that $D_{1,k+1}^+<D_{1,k}^+$
  \end{case}

  All that remains to show is that $D_{1,k+1}^+<D_{1,k}^+$. For $C_1^+$ we will consider the following cases:

\setcounter{casecount}{0}
  \begin{case}
  Suppose there is some $j\neq 1$ such that $c_{1,j}^+(k)\neq x_j(k)$ for some $j\neq 1$, i.e. $C_1^+$ has not coordinate matched the robber. Without loss of generality suppose $c_{1,j}^+(k)>x_j(k)$. Then according to our strategy, $c_{1,j}^+(k+1)=c_{1,j}^+(k)-1$ and $c_{1,i}^+(k)=c_{1,i}^+(k)$ for all $i\neq j$. Therefore:
     \setlength{\abovedisplayskip}{5pt}
\setlength{\belowdisplayskip}{5pt}
\setlength{\abovedisplayshortskip}{0pt}
\setlength{\belowdisplayshortskip}{0pt}
   \begin{align*}
     D_{1,k+1}^+= &\sum_{m=1}^n |c_{1,m}^{+}(m) - x_m(k+1)|\\
     &=|c^+_{1,j}(k+1)-x_j(k+1)|+\sum_{m \neq j}^n |c_{1,m}^{+}(k+1) - x_m(k+1)|\\
     &=|c^+_{1,j}(k)-1-x_j(k)|+|c_{1,1}^+(k)-(x_1(k)+1)| +\sum_{m \neq 1,j}^n |c_{1,m}^{+}(k) - x_m(k)|\\
     &=c^+_{1,j}(k)-x_j(k)-1+c_{1,1}^+(k)-x_1(k) -1+\sum_{m \neq 1,j}^n |c_{1,m}^{+}(k) - x_m(k)|\\
     &=|c^+_{1,j}(k)-x_j(k)|-1+|c_{1,1}^+(k)x_1(k) |-1+\sum_{m \neq 1,j}^n |c_{1,m}^{+}(k) - x_m(k)|\\
     &=\sum_{m=1}^n |c_{1,m}^{+}(k) - x_m(k)|-2\\
     &<D_{1,k}^+   
\end{align*}

    \end{case}

  \begin{case}
      Suppose $C_1^-$ has coordinate matched the robber. Then $x_j(k)=c_{1,j}^-(k)$ for all $j\neq 1$. If $x_1(k)=c_{1,1}^-(k)$ as well, then the robber has been caught. So let us suppose $x_1(k)> c_{1,1}^-(k)$. Then according to our strategy, $c_{1,1}^-(k+1)=c_{1,1}^-(k)+1$. A computation similar to the one above easily shows that $D_{1,k+1}^-=D_{1,k}^-$
  \end{case}
  
Therefore the inductive hypothesis holds and the strategy must eventually terminate, as $D_{i,k}^\pm$ will eventually be 0 for at least one $i\in \{1,\dots,n\}$ and for some $k\in \mathbb{N}$.

\end{proof}

\section{Classification of Winning Sets in $\mathbb{Z}^n$}

We now have everything we need to provide a classification of sets of winning cops played on $\mathbb{Z}^n$.

\begin{theorem}\label{general result with Vs}
    A collection of cops $\mathcal{C}$ in $\mathbb{Z}^n$ can catch the robber if and only if for each $m\in \{1,...,n\}$ there exist infinitely many $a\in \mathbb{N}$ such that there is at least one cop in $V_{m,a}^+$ and infinitely $b\in \mathbb{N}$ such that there is at least one cop in $V_{m,b}^-$.
\end{theorem}

\begin{proof}
    Suppose by way of contraposition that there is some $m\in \{1,\dots,n\}$ such that there are only finitely many $a\in\mathbb{N}$ with at least one cop in $V_{m,a}^+$ or in $V_{m,a}^-$. Without loss of generality due to symmetry, suppose that $m=1$ and that there are only finitely many positive integers  $a_1,\dots,a_k$ with at least one cop in $V_{m,a_l}^+$ for $1\leq l\leq k$. Let $a'=\max\{a_1,\dots,a_k\}$. By assumption, there are no cops in $V_{1,b}^+$ for any $b>a'$. Let the robber's starting position be $R(0)=(a'+2,0,\dots,0)$ and define the robber's strategy $\sigma$ to be to move along the positive direction of the $X_1$-axis each turn, i.e. at the end of the robber's $k$th turn, his position will be given by $R(k)=(a'+2+k,0,\dots,0)$.\\

    Suppose by way of contradiction that the robber is caught while following $\sigma$. Then there exists some $k\in \mathbb{N}$ such that at the end of the cops' turn, the robber is caught on $(a'+2+k,0,\dots,0)$. First, note that $k\geq 2$, since for any $C\in \mathcal{C}$, $|C(0)-R(0)|\geq 2$. Next, note that any cop whose $x_1$-coordinate is less than $a'+4$ cannot catch the robber, since for any $C\in \mathcal{C}$, $C\notin V_{1,b}^+$ for any $b>a'$. Thus the only cops who could potentially catch the robber are cops whose starting coordinates, say $(x_1,x_2,\dots,x_n)$ satisfy:
    \begin{center}
        $x_1\geq a'+4$ and $|x_j|\geq 4$ for some $i\in \{2,\dots,n\}$.
    \end{center}

    Suppose that the robber is caught on $(a'+2+k,0,\dots,0)$ by $C\in \mathcal{C}$, whose starting position is given by $C(0)=(x_1,x_2,\dots,x_n)$. Then we must have
    \begin{center}
        $|x_1-(a'+2+k)|+\displaystyle{\Sigma_{j=2}^n|x_j-0|}\leq k$\\
    \end{center}

    Since $x_1\geq a'+4$, we have

        \begin{align*}            
        &|a'+4-(a'+2+k)|+\displaystyle{\Sigma_{j=2}^n|x_j-0|}\leq k\\       
        \implies &|2-k|+\displaystyle{\Sigma_{j=2}^n|x_j|}\leq k\\     
        \implies &k-2+\displaystyle{\Sigma_{j=2}^n|x_j-0|}\leq k \text{ since } k\leq 2\\    
        \implies &\displaystyle{\Sigma_{j=2}^n|x_j-0|}\leq 2\\  
        \implies & 4\leq |x_j| \leq \displaystyle{\Sigma_{j=2}^n|x_j-0|}\leq 2\\
    \end{align*}
    \vspace{3mm}
    A contradiction, thus following the strategy $\sigma$ the robber cannot be caught. Therefore if a collection of cops catch the robber then for each $m\in \{1,...,n\}$ there exist infinitely many $a\in \mathbb{N}$ such that there is at least one cop in $V_{{m,a}}^+$ and infinitely $b\in \mathbb{N}$ such that there is at least one cop in $V_{{m,b}}^-$.\\

Next, suppose for each $m\in \{1,...,n\}$ there exist infinitely many $a\in \mathbb{N}$ such that there is at least one cop in $V_{{m,a}}^+$ and infinitely $b\in \mathbb{N}$ such that there is at least one cop in $V_{{m,b}}^-$. Let the robber's starting position be $R(0)=(x_1,x_2,\dots,x_n)$. Since the set $\{V_{i,m}^+,V_{i,m}^-\}_{i=1,m=0}^{n,\infty}$ satisfies $\displaystyle{\cup_{m=0}^\infty \cup_{i=1}^n V_{i,m}^\pm}=\mathbb{Z}^n$, there exists some $i\in \{1,\dots,n\}$ and some $m\in \mathbb{N}$ such that $R(0)\in V_{i,m}^+$ or $R(0)\in V_{i,m}^-$. Without loss of generality due to symmetry, suppose $i=1$, $R(0)\in V_{1,m}^+$, and $x_j\geq 0$ for each $1\leq j\leq n$. For $1\leq i\leq n$ define $V_{i,R}^+$ to be $\{(y_1+x_1,y_2+x_2,\dots,y_x+x_n):(y_1,y_2,\dots,y_n)\in V_{i,0}^+\}$, in other words $V_{i,R}^+$ is $V_{i,0}^+$ shifted so that the robber's starting position is viewed as the origin. Define $V_{i,R}^-$ to be $\{(y_1+x_1,y_2+x_2,\dots,y_x+x_n):(y_1,y_2,\dots,y_n\in V_{i,0}^-\}$. We consider the following cases for the robber's starting position:\\

\begin{case}
    Suppose $R(0)=(0,0,\dots,0)$. By assumption, for each $1\leq i\leq n$ there exists minimal values $ k(i)^+,k(i)^-\in \mathbb{N}_{\geq 1}$ such that there is at least one cop in $V_{i,k(i)^+}^+$ and at least one cop in $V_{i,k(i)^-}^-$. Let $C_i^+$ be any cop in $V_{i,k(i)^+}^+$ and $C_i^-$ be any cop in $V_{i,k(i)^-}^-$. By Lemma \ref{Origin catch in n}, $\mathcal{C}=\{C_i^+,C_i^-:1\leq i \leq n\}$ can catch the robber. \\
\end{case}

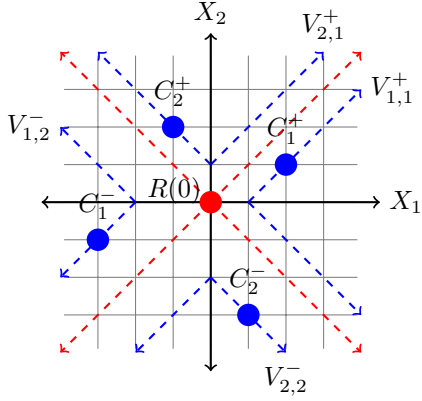
\begin{figure}
\begin{center}
\begin{tikzpicture}[scale=0.5]

\draw[step=1cm,gray,thin] (-3.9,-3.9) grid (3.9,3.9);

\draw[thick,->] (0,0) -- (4.5,0);
\draw[thick,->] (0,0) -- (0,4.5);
\draw[thick,<-] (-4.5,0) -- (0,0);
\draw[thick,<-] (0,-4.5) -- (0,0);

\draw[red,thick,dashed,->] (0,0) -- (4,4);
\draw[red,thick,dashed,->] (0,0) -- (-4,4);

\draw[red,thick,dashed,->] (0,0) -- (-4,-4);
\draw[red,thick,dashed,->] (0,0) -- (4,-4);

\draw[blue,thick,dashed,->] (0,1) -- (-3,4);
\draw[blue,thick,dashed,->] (0,1) -- (3,4);
\draw(3,4) node[anchor=south] {$V_{2,1}^+$};

\draw[blue,thick,dashed,->] (0,-2) -- (-2,-4);
\draw[blue,thick,dashed,->] (0,-2) -- (2,-4);
\draw(2,-4) node[anchor=north] {$V_{2,2}^-$};

\draw[blue,thick,dashed,->] (-2,0) -- (-4,-2);
\draw[blue,thick,dashed,->] (-2,0) -- (-4,2);
\draw(-4,2) node[anchor=east] {$V_{1,2}^-$};

\draw[blue,thick,dashed,->] (1,0) -- (4,3);
\draw[blue,thick,dashed,->] (1,0) -- (4,-3);
\draw(4,3) node[anchor=west] {$V_{1,1}^+$};

\node[circle,fill=red, inner sep=3pt,label={[anchor=east]$R(0)$}] at (0,0){};

\node[circle,fill=blue, inner sep=3pt,label=$C_1^+$] at (2,1){};
\node[circle,fill=blue, inner sep=3pt,label=$C_1^-$] at (-3,-1){};
\node[circle,fill=blue, inner sep=3pt,label=$C_2^+$] at (-1,2){};
\node[circle,fill=blue, inner sep=3pt,label=$C_2^-$] at (1,-3){};

\draw(4.5,0) node[anchor=west] {$X_{1}$};
\draw(0,4.5) node[anchor=south] {$X_{2}$};

\end{tikzpicture}
\caption{An example of Case 1 in $\mathbb{Z}^2$.}
\end{center}
\end{figure}

\begin{case}
   Suppose $R_0=(a,a,\dots,a)$ for some $a\in \mathbb{N}_{n\geq 1}$. Let $\hat{k}(1)^+$ be such that $V_{1,R}^+\cap \{(x_1,0,\dots,0):x_1\in \mathbb{Z}^+\}=\{(\hat{k}_1^+,0,\dots,0)\}$. Note that $\hat{k}(1)^+>0$. By assumption, there exists a least $k(1)^+\in \mathbb{N}$ such that $k(1)^+\geq \hat{k}(1)^+$ and there is at least one cop in $V_{1,k(1)^+}^+$. Let $C_1^+$ be any cop in $V_{1,k(1)^+}^+$. Let $k(1)-\in \mathbb{N}$ ne least such that there is at least one cop in $V_{1,k(1)^-}^-$. Let $C_1^-$ be any cop in $V_{1,k(1)^-}^-$. For each $2\leq i \leq n$, define $\hat{k}(i)^+$ to be such that $V_{i,R}^+\cap X_i=\{(x_1,x_2,\dots,x_n):x_i=\hat{k}(i)^+,x_j=0 \text{ for all }j\neq i\}$. For each $2\leq i \leq n$ let $k(i)^+$ be the least integer such that $k(i)^+\geq \hat{k}(i)^+$ and there is at least one cop in $V_{i.k(i)^+}^+$. Let $C_i^+$ be any cop in $V_{i.k(i)^+}^+$. For each $2\leq i \leq n$, let $k(i)^-\in \mathbb{N}$ be least such that there is at least one cop in $V_{i,k(i)^-}^-$. Let $C_i^-$ be any cop in $V_{i,k(i)^-}^-$. By Lemma \ref{Origin catch in n}, $\mathcal{C}=\{C_i^+,C_i^-:1\leq i \leq n\}$ can catch the robber. \\
\end{case}

\begin{figure}
\begin{center}
\begin{tikzpicture}[scale=0.5]

\draw[step=1cm,gray,thin] (-3.9,-3.9) grid (3.9,3.9);

\draw[thick,->] (0,0) -- (4.5,0);
\draw[thick,->] (0,0) -- (0,4.5);
\draw[thick,<-] (-4.5,0) -- (0,0);
\draw[thick,<-] (0,-4.5) -- (0,0);

\draw[blue,thick,dashed,->] (0,2) -- (-2,4);
\draw[blue,thick,dashed,->] (0,2) -- (2,4);
\draw(2,4) node[anchor=south] {$V_{2,2}^+$};

\draw[blue,thick,dashed,->] (0,-2) -- (-2,-4);
\draw[blue,thick,dashed,->] (0,-2) -- (2,-4);
\draw(2,-4) node[anchor=north] {$V_{2,2}^-$};

\draw[blue,thick,dashed,->] (-2,0) -- (-4,-2);
\draw[blue,thick,dashed,->] (-2,0) -- (-4,2);
\draw(-4,2) node[anchor=east] {$V_{1,2}^-$};

\draw[blue,thick,dashed,->] (3,0) -- (4,1);
\draw[blue,thick,dashed,->] (3,0) -- (4,-1);
\draw(4,1) node[anchor=south] {$V_{1,3}^+$};

\draw[red,thick,dashed,->] (1,1) -- (4,4);
\draw[red,thick,dashed,->] (1,1) -- (-2,4);

\draw[red,thick,dashed,->] (1,1) -- (4,-2);
\draw[red,thick,dashed,->] (1,1) -- (-4,-4);

\node[circle,fill=red, inner sep=3pt,label=$R(0)$] at (1,1){};

\node[circle,fill=blue, inner sep=3pt,label=$C_1^+$] at (3,0){};
\node[circle,fill=blue, inner sep=3pt,label=$C_1^-$] at (-3,-1){};
\node[circle,fill=blue, inner sep=3pt,label=$C_2^+$] at (-1,3){};
\node[circle,fill=blue, inner sep=3pt,label=$C_2^-$] at (-1,-3){};

\draw(4.5,0) node[anchor=west] {$X_{1}$};
\draw(0,4.5) node[anchor=south] {$X_{2}$};

\end{tikzpicture}
\caption{An example of Case 2 in $\mathbb{Z}^2$.}
\end{center}
\end{figure}
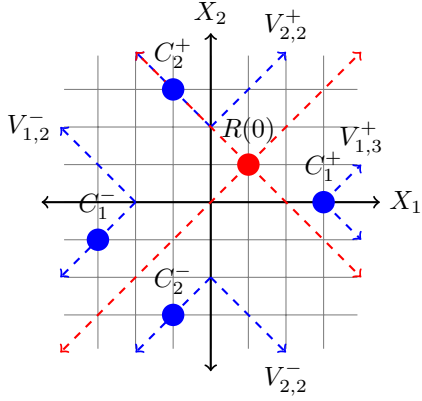

\begin{case} 
    Suppose $R(0)=(x_1,0,\dots,0)$ for some $x_1\in \mathbb{N}_{\geq 1}$. Let $k(1)^+\in \mathbb{N}$ be least such that $k(1)^+\geq x_1$ and there is at least one cop in $V_{1,k(1)^+}^+$. Let $C_1^+$ be any cop in $V_{1,k(1)^+}^+$. Let $K(1)^-\in \mathbb{N}$ be least such that there is at least one cop in $V_{1,K(1)^-}^-$. Let $C_1^-$ be any cop in $V_{1,K(1)^-}^-$. For each $2\leq i\leq n$ define $\hat{k}(i)^+$ to be such that $V_{i,R}^+\cap X_i=\{(x_1,x_2,\dots,x_n):x_i=\hat{k}(i)^+,x_j=0 \text{ for all }j\neq i\}$. For each $2\leq i \leq n$ let $k(i)^+$ be the least integer such that $k(i)^+\geq \hat{k}(i)^+$ and there is at least one cop in $V_{i.k(i)^+}^+$. Let $C_i^+$ be any cop in $V_{i.k(i)^+}^+$ and pick $C_i^-$ similarly.  By Lemma \ref{Origin catch in n}, $\mathcal{C}=\{C_i^+,C_i^-:1\leq i \leq n\}$ can catch the robber. \\
\end{case}

\begin{case}
    Suppose $R(0)=(x_1,x_2,\dots,x_n)$ where $x_i\neq x_j$ for $1\leq i<j\leq n$ and $x_k\neq 0$ for some $1\leq k \leq n$. For each $1\leq i\leq n$ define $\hat{k}(i)^+$ to be such that $V_{i,R}^+\cap X_i=\{(x_1,x_2,\dots,x_n):x_i=\hat{k}(i)^+,x_j=0 \text{ for all }j\neq i\}$. For each $1\leq i \leq n$ let $k(i)^+$ be the least integer such that $k(i)^+\geq \hat{k}(i)^+$ and there is at least one cop in $V_{i.k(i)^+}^+$. Let $C_i^+$ be any cop in $V_{i.k(i)^+}^+$ and pick $C_i^-$ similarly.  By Lemma \ref{Origin catch in n}, $\mathcal{C}=\{C_i^+,C_i^-:1\leq i \leq n\}$ can catch the robber. \\
\end{case}

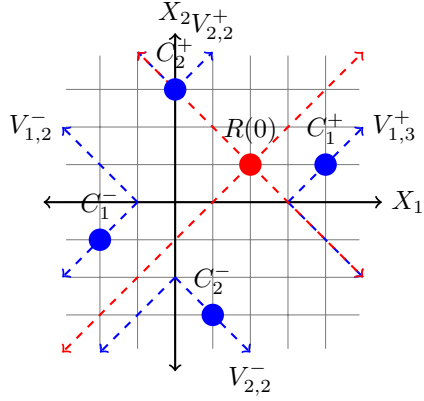
\begin{figure}
\begin{center}
\begin{tikzpicture}[scale=0.5]

\draw[step=1cm,gray,thin] (-3.9,-3.9) grid (3.9,3.9);

\draw[thick,->] (0,0) -- (4.5,0);
\draw[thick,->] (-1,0) -- (-1,4.5);
\draw[thick,<-] (-4.5,0) -- (0,0);
\draw[thick,<-] (-1,-4.5) -- (-1,0);

\draw[blue,thick,dashed,->] (-1,3) -- (-2,4);
\draw[blue,thick,dashed,->] (-1,3) -- (0,4);
\draw(0,4) node[anchor=south] {$V_{2,2}^+$};

\draw[blue,thick,dashed,->] (-1,-2) -- (1,-4);
\draw[blue,thick,dashed,->] (-1,-2) -- (-3,-4);
\draw(1,-4) node[anchor=north] {$V_{2,2}^-$};

\draw[blue,thick,dashed,->] (-2,0) -- (-4,-2);
\draw[blue,thick,dashed,->] (-2,0) -- (-4,2);
\draw(-4,2) node[anchor=east] {$V_{1,2}^-$};

\draw[blue,thick,dashed,->] (2,0) -- (4,2);
\draw[blue,thick,dashed,->] (2,0) -- (4,-2);
\draw(4,2) node[anchor=west] {$V_{1,3}^+$};

\draw[red,thick,dashed,->] (1,1) -- (4,4);
\draw[red,thick,dashed,->] (1,1) -- (-2,4);

\draw[red,thick,dashed,->] (1,1) -- (4,-2);
\draw[red,thick,dashed,->] (1,1) -- (-4,-4);

\node[circle,fill=red, inner sep=3pt,label=$R(0)$] at (1,1){};

\node[circle,fill=blue, inner sep=3pt,label=$C_1^+$] at (3,1){};
\node[circle,fill=blue, inner sep=3pt,label=$C_1^-$] at (-3,-1){};
\node[circle,fill=blue, inner sep=3pt,label=$C_2^+$] at (-1,3){};
\node[circle,fill=blue, inner sep=3pt,label=$C_2^-$] at (0,-3){};

\draw(4.5,0) node[anchor=west] {$X_{1}$};
\draw(-1,4.5) node[anchor=south] {$X_{2}$};

\end{tikzpicture}
\caption{An example of Case 4 in $\mathbb{Z}^2$.}
\end{center}
\end{figure}

Therefore given any starting position, the robber will always be caught.
\end{proof}

\begin{corollary}
    Given any starting collection of robbers in $\mathbb{Z}^n$, if for each $m\in \{1,...,n\}$ there exist infinitely many $a\in \mathbb{N}$ such that there is at least one cop in $V_{m,a}^+$ and infinitely $b\in \mathbb{N}$ such that there is at least one cop in $V_{m,b}^-$ then the collection of cops $\mathcal{C}$ can catch the robbers.
\end{corollary}

\begin{proof}
    Since any collection of robbers in $\mathbb{Z}^n$ must be countable, we can enumerate the set of robbers: $\mathcal{R}=\{R_1,R_2,\dots\}$. Now consider the first robber in $\mathcal{R}$, $R_1$. Following the same steps as outlined in Theorem \ref{general result with Vs}, regardless of the robber's starting position there exists a distinct cop in each of $V_{i,R_1}^+$ and $V_{i,R_1}^-$ for each $1\leq i\leq n$. Let us denote this collection of $2n$ distinct cops as $\mathcal{C}_1$. Then by Theorem \ref{general result with Vs}, the cops in $\mathcal{C}_1$ can catch $R_1$ in finitely many steps.. Now consider the second robber $R_2$. By the same argument as in Theorem \ref{general result with Vs}, there exists $2n$ distinct cops, one in each of $V_{i,R_2}^\pm$ for $1\leq i\leq n$. In particular, we may pick this collection of cops to be entirely disjoint from $\mathcal{C}_1$, as by assumption there are infinitely many cops placed appropriately throughout $\mathbb{Z}^n$. Thus there is a set of cops $\mathcal{C}_2$ entirely disjoint from $\mathcal{C}_1$ such that $\mathcal{C}_2$ can catch $R_2$ by Theorem \ref{general result with Vs}. Similarly, for each robber $R_k\in \mathcal{R}$, we can find a collection of cops $\mathcal{C}_k$ that is disjoint from $\mathcal{C}_i$ for all other positive integers $i\neq k$ such that the cops in $\mathcal{C}_k$ can capture $R_k$. 
\end{proof}

This concludes our classification of the infinite variation of Cops and Robbers in $\mathbb{Z}^n$. A natural continuation of this research would be to investigate different variations of the Cops and Robbers game in $\mathbb{Z}^n$. For example, one interesting variation of Cops and Robbers played in $\mathbb{Z}^n$ involves the robber defining a fixed radius about his starting position, inside which no cops are allowed to be placed. The game begins with the robber picking his starting coordinate in $\mathbb{Z}^n$ and a non-negative integer, $R(0)=(x_1,x_2,...,x_n)$ and $r$ respectively. The cops are then free to place themselves on any coordinate except inside $R=\{\overline{y}\in \mathbb{Z}^n: |\overline{y}-R(0)|\leq r\}$. This variation of the game forms a basis for research currently being pursued by our team.

\newpage
\section{Funding and Declarations}
The authors declare that no funds, grants, or other support were received during the preparation of this manuscript. All authors certify that they have no affiliations with or involvement in any organization or entity with any financial interest or non-financial interest in the subject matter or materials discussed in this manuscript.

\end{document}